\documentclass[a4paper,12pt]{amsart}

\usepackage{amsmath,amsthm,amssymb}
\usepackage{hyperref}

\allowdisplaybreaks[1]

\newcommand{\al}{\alpha}
\newcommand{\be}{\beta}
\newcommand{\la}{\lambda}
\newcommand{\ity}{\infty}
\newcommand{\C}{\mathbb{C}}

\newcommand{\N}{\mathbb{N}}
\newcommand{\Z}{\mathbb{Z}}

\newcommand{\B}{\Big}

\numberwithin{equation}{section}
\newtheorem{theorem}{Theorem}[section]
\newtheorem{lemma}[theorem]{Lemma}
\newtheorem{proposition}[theorem]{Proposition}
\newtheorem{corollary}[theorem]{Corollary}
\theoremstyle{remark}
\newtheorem{remark}[theorem]{Remark}
\newtheorem{example}[theorem]{Example}

\begin{document}
\title[escaping set of an entire function]{ results on escaping set of an entire function and its composition}

%

\author[R. Kaur]{Ramanpreet Kaur}
\address{Department of Mathematics, University of Delhi,
Delhi--110 007, India}

\email{preetmaan444@gmail.com}

\author[D. Kumar]{Dinesh Kumar}
\address{Department of Mathematics, Deen Dayal Upadhyaya College, University of Delhi,
Delhi--110 078, India}

\email{dinukumar680@gmail.com }

%
%

\begin{abstract}
Given two permutable entire functions $f$ and $g,$ we establish vital relationship between escaping sets of entire functions $f, g$ and their composition. We provide some families of transcendental entire functions for which Eremenko's conjecture holds. In addition, we investigate the dynamical properties of the mapping $f(z)=z+1+e^{-z}.$
\end{abstract}



\maketitle
\section{introduction}\label{sec1}

 Let $f$ be a transcendental entire function. For $n\in\N$ let $f^n$ denote the $n$-th iterate of $f.$ 

The set $F(f)=\{z\in\C : \{f^n\}_{n\in\N}\text{ is normal in some neighborhood of }z\}$
is called the Fatou set of $f$ or the set of normality of $f.$ Its complement $J(f)$ is the Julia set of $f$.
 For an introduction to the basic properties of these sets the reader can refer to \cite{bergweiler}.  The escaping set of $f$ denoted by $I(f)$ is the collection of those  points in the complex plane that tend to infinity under iteration of $f$. In general, it is neither an open nor a closed subset of $\C$ and has interesting topological properties.  The escaping set of a transcendental entire function $f$ was studied for the first time by Eremenko ~\cite {e1} who established that
\begin{enumerate}
\item\ $I(f) \neq\emptyset;$
\item\ $J(f)=\partial I(f);$
\item\ $I(f)\cap J(f)\neq\emptyset;$
\item\ $\overline{I(f)}$ has no bounded components.
\end{enumerate}
In the same paper he stated the following conjectures:
\begin{enumerate}
\item[(i)] Every component of $I(f)$ is unbounded;
\item[(ii)] Every point of $I(f)$ can be connected to $\ity$ by a curve consisting of escaping points.
\end{enumerate}

For the exponential maps of the form $f(z)=e^z+\la$ with $\la>-1,$ it is known, by Rempe \cite{R2}, that the escaping set is a connected subset of the plane, and for $\la<-1,$ it is the disjoint union of uncountably many curves to infinity, each of which is connected component of $I(f)$ \cite{R4}. 
In ~\cite{SZ} it was shown that every escaping point of every exponential map can be connected to $\ity$ by a curve consisting of escaping points.


Recall that $w\in\C$ is a critical value of a transcendental entire function $f$ if there exist some $w_0\in\C$ with $f(w_0)=w$ and $f'(w_0)=0.$ Here $w_0$ is called a critical point of $f.$ The image of a critical point of $f$ is  called critical value of $f.$ Also  $\xi\in\C$ is an asymptotic value of a transcendental entire function $f$ if there exist a curve $\Gamma$ tending to infinity such that $f(z)\to \xi$ as $z\to\ity$ along $\Gamma.$ The set of critical values and asymptotic values of $f$ will be denoted by $CV(f)$ and $AV(f)$ respectively.

For the Eremenko-Lyubich class \[\mathcal{B}=\{f:\C\to\C\,\,\text{transcendental entire}: \text{Sing}{(f^{-1})}\,\text{is bounded}\}\]
\,(where Sing($f^{-1}$) is the set of critical values and asymptotic values of $f$ and their finite limit points), it is well known that $I(f)$ is a subset of the Julia set $J(f),$  \cite{EL}. Each $f\in\mathcal B$ is said to be of bounded type. Also  if $f$ and $g$ are of bounded type, then so is $f\circ g$ \cite{berg2}, i.e. functions in $\mathcal B$ are closed under taking composition. 

A set $W$ is forward invariant under $f$ if $f(W)\subset W$  and $W$ is backward invariant under $f$ if $f^{-1}(W)=\{w\in\C:f(w)\in W\}\subset W.$ Moreover, $W$ is called completely invariant under $f$ if it is both forward and backward invariant under $f.$ We will be dealing mostly with permutable entire functions where 
two functions $f$ and $g$ are called permutable or commuting if $f\circ g=g\circ f.$ 

Recently, the dynamics of composite of two or more  complex functions have been studied by many authors.
 The seminal work in this direction was done by Hinkkanen and Martin \cite{martin} related to semigroups of rational functions. In their  papers, they extended the classical theory of the dynamics  associated to the iteration of a rational function of one complex variable to a more general setting of an arbitrary semigroup of rational functions. Several results were extended to semigroups of transcendental entire functions in \cite{{cheng}, {dinesh1}, {dinesh5}, {dinesh6}, {poon1}, {zhigang}}. The second author \cite{dinesh5} established several classes of transcendental semigroups for which Eremenko's conjecture holds.

In this article, we have investigated relationship between escaping sets of two permutable entire functions. To our best knowledge, it has not been studied so far what can be said about relation between escaping set of entire functions $f, g$ and their compositions. Our aim in this paper is to tackle this problem and establish several such relationships. We have shown that postsingularly bounded entire functions and hyperbolic entire functions (to be defined in Section 2) are closed under taking compositions. We provide families of transcendental entire functions for which Eremenko's conjecture holds. In addition, we investigate some dynamical properties of the mapping $f(z)=z+1+e^{-z}.$

\section{results}\label{sec2}

In this section, we establish several relationship between escaping sets of two permutable entire functions $f, g$ and their compositions.
We  first establish an important relation between $I(f), I(g)$ and $I(f\circ g)$ for two permutable entire functions $f$ and $g.$

\begin{theorem}\label{sec2,thm4}
Let $f$ and $g$ be   transcendental  entire functions satisfying $f\circ g=g\circ f.$ Then
\begin{enumerate}
\item [(i)] $I(f\circ g)$ is completely invariant under $f$ and $g$;
\item [(ii)] $I(f\circ g)\subset I(f)\cup I(g);$
\item [(iii)] For any two positive integers $i$ and $j,$ $I(f^i\circ g^j)\subset I(f)\cup I(g);$
\item [(iv)] For any two positive integers $i$ and $j,$ $I(f^i\circ g^j)\subset I(f\circ g)\).
\end{enumerate}
\end{theorem}
\begin{proof}
\begin{enumerate}
\item [(i)] 
We first show that $z\in I(f\circ g)$ if and only if $g(z)\in I(g\circ f).$ Let $z\in I(f\circ g).$ Then $(f\circ g)^n(z)\to\ity$ as $n\to\ity,$ i.e.,
$f((g\circ f)^{n-1}g(z))\to\ity$ as $n\to\ity.$ As $f$ is an entire function, this implies that $(g\circ f)^{n-1}g(z)\to\ity$ as $n\to\ity,$ i.e., $g(z)\in I(g\circ f).$ On the other hand, let $g(z)\in I(g\circ f).$ Then $(g\circ f)^{n}(g(z))\to\ity$ as $n\to\ity,$ i.e., $g((f\circ g)^n(z))\to\ity$ as $n\to\ity.$ Again, as $g$ is entire, this forces $(f\circ g)^n(z)\to\ity$ as $n\to\ity.$ We deduce that $z\in I(f\circ g)$ which proves the claim. As $f\circ g=g\circ f,$ we obtain $z\in I(f\circ g)$ if and only if $g(z)\in I(f\circ g)$ which implies $I(f\circ g)$ is completely invariant under $g,$ and by symmetry, under $f$ as well.
\item [(ii)] Suppose that \(z\not\in I(f)\) and \(z\not \in I(g)\). We shall prove that \(z\not \in I(f\circ g)\). \\Consider the following infinite sequences of rows where each row consists of orbit of $\{g^k(z): k\in\N\}$ under $f:$

\[\begin{array}{ll}
\vspace{3mm}
 f(g(z)) , f^2(g(z)) ,f^3(g(z)),\ldots\\
\vspace{3mm} 
f(g^2(z)) , f^2(g^2(z)) ,f^3(g^2(z)),\ldots\\
\vspace{3mm}
 f(g^3(z)) , f^2(g^3(z)) ,f^3(g^3(z)),\ldots\\
\vspace{3mm}
  \quad \vdots \quad\quad\quad \quad \quad \vdots\quad \quad \quad \quad  \quad \vdots\\
f(g^k(z)) ,f^2(g^k(z)) ,f^3(g^k(z)),\ldots\\
  \quad \vdots \quad\quad\quad \quad \quad \vdots\quad \quad \quad \quad  \quad \vdots
     \end{array}\]


 Since \(z\not\in I(f)\) and \(z\not \in I(g)\), therefore each of the above sequence is bounded and hence has a limit point. We can assume that each of the above sequence is actually convergent. Let \(l\) be the limit  of the sequence corresponding to first row, i.e., \(f^n(g(z))\to l\) as \(n\to \infty\) which gives \( f^n(g^2(z)))\to g(l)\) as \(n\to \infty\). Proceeding on similar lines, we have  \(f^n(g^k(z)))\to g^{k-1}(l)\) as \(n\to \infty.\) Let 
\[X=\{l,g(l),g^2(l)\ldots\}.\]
 We have the following cases to consider.\\ 
Case 1: \(l= g(l)\) i.e., \(l\) is a fixed point of \(g\), then \(g^k(l)=l , \forall\,  k \in\N.\)   In particular, we can extract a subsequence \(\{f^n(g^n(z)), n\in \N\}\) which also converges to \(l\),  i.e., \(\lim\limits_{n\to \infty}f^n(g^n(z))=l(\neq \infty)\) from which we see that \( z\not \in I(f\circ g)\).
\\Case 2: \(g^m(l)=g^{m+j}(l)\) for some  \(m\in \mathbf{N}\) and $j\geq 1$ i.e., \(g^m(l)\) is a fixed point of \(g\), which in turn implies
\begin{equation}\notag
\begin{split}
 \lim\limits_{n\to \infty}f^n(g^{m+j}(z))
&=g^{m+j}(l)\\
&=g^m(l), j\geq 1.
\end{split}
\end{equation}
This gives  \( \lim\limits_{n\to\infty}f^n(g^k(z))=g^m(l), \forall\, k\geq m\). Therefore, \(\lim\limits_{n\to \infty}f^n(g^n(z))=g^m(l)(\neq \infty)\) from which we deduce that  \(z\not \in I(f\circ g)\).\\ 
Case 3: When \(X\) is an infinite set then it  has a limit point \(\alpha\, \mathit{(say)}\)  
     i.e., \(g^{k-1}(l)\to \alpha\) as \(k\to \infty\). We consider the following two subcases:\\
     Subcase 1: \(\alpha\) is finite.\\
     As we have \(f^n(g(z))\to l\) as \(n\to \infty\), so by continuity of \(g\) we have
     \(\lim\limits_{n\to \infty}f^n(g^k(z))= g^{k-1}(l)\). This implies that  \( \lim\limits_{k\to \infty}\lim\limits_{n\to \infty}f^n(g^k(z))= \alpha\). Thus, every subsequence  will also converge to \(\alpha\), i.e., \(\lim\limits_{n\to \infty}f^n(g^n(z))=\alpha\) so that \( z\not\in I(f\circ g).\)\\
     Subcase 2: \(\alpha\) is infinity.\\
     By using  similar arguments, we have \(\lim\limits_{k\to\infty}\lim\limits_{n\to \infty}f^n(g^k(z))=\infty\). In particular, \(\lim\limits_{n\to\infty}f^n(g^k(z))=\infty\), i.e. \(g^k(z)\in I(f)\) which implies that \(z\in I(f) \), which is a contradiction to our supposition. 
This completes the proof of the result.

\item[(iii)] For $i, j\in\N,$ assume $i\geq j.$ Now using previous result we have
\begin{equation}\notag
\begin{split}
 I(f^i\circ g^j)
&=I(f^{i-j}\circ {f^j\circ g^j})\\
&\subset I(f^{i-j})\cup I(f^j\circ g^j)\\
&=I(f)\cup I(f\circ g)\\
&\subset I(f)\cup I(g).
\end{split}
\end{equation}
\item[(iv)] Suppose that  $i> j.$ Then $j+k=i$ for some $k\in\N$. Let $z\in I(f^i\circ g^j),$ i.e., $(f^i\circ g^j)^n(z)\to\infty$ as $n\to\infty.$
Suppose that $z\not\in I(f\circ g),$ i.e., $(f\circ g)^n(z)\to l$ (say), as $n\to \infty.$ Consider the following infinite sequences of rows where first row consists of orbit of $z$ under $f\circ g$ and the successive rows are the images of this orbit under $f^k, k\in\N$ respectively:\\
\[ \begin{array}{llllllllll}
	f(g(z)),f^2(g^2(z)),\ldots,f^j(g^j(z)),\ldots,f^i(g^i(z)),\ldots,f^{2j}(g^{2j}(z)),\ldots\\
\vspace{3mm}
	   f^2(g(z)),f^3(g^2(z)),\ldots,f^{j+1}(g^j(z)),\ldots,f^{i+1}(g^i(z)),\ldots,f^{2j+1}(g^{2j}(z)),\ldots\\
\vspace{3mm}
	  \quad \vdots\qquad\qquad\qquad \qquad\quad\quad \quad \quad \vdots\qquad\qquad\qquad\qquad \quad \quad \quad  \quad \vdots\\
\vspace{3mm}
	    f^{1+k}(g(z)),f^{2+k}(g^2(z)),\ldots,f^{j+k=i}(g^j(z)),\ldots,f^{i+k}(g^i(z)),\ldots,f^{2j+k}(g^{2j}(z)),\ldots\\
\vspace{3mm}
	   \quad \vdots \qquad\qquad\qquad \qquad \quad\quad\quad\quad \vdots\qquad \qquad \qquad \qquad  \quad\quad\quad\quad \vdots\\
\vspace{3mm}
	    f^{1+2k}(g(z)),f^{2+2k}(g^2(z)),\ldots,f^{i+k}(g^j(z)),\ldots,f^{i+2k}(g^i(z)),\ldots,f^{2j+2k=2i}(g^{2j}(z)),\ldots\\
 \quad \vdots \qquad\qquad\qquad \qquad \quad\quad\quad
\quad \vdots\qquad \qquad \qquad \qquad  \quad\quad\quad\quad \vdots
	   \end{array} \]

The first row converges to $l$. Using continuity of $f,$ second row will converge to $f(l)$ and so on. Let 
\[\{l,f(l),\ldots\}\]
 be the collection of  limit points corresponding to the rows. Hence has a limit point  $\alpha $ (say).\\
	   Case 1: \(\alpha(\neq\infty)\)\\
	   i.e. \(\lim\limits_{r\to\infty}f^{r-1}(l)=\alpha\).
	   As \(\lim\limits_{n\to \infty}f^n(g^n(z))=l\), so 
	   \( \lim\limits_{n\to\infty}f^{n+r-1}(g^n(z))= f^{r-1}(l)\). This gives
\begin{equation}\notag
\begin{split}
 \lim\limits_{r\to\infty} \lim\limits_{n\to\infty}f^{n+r-1}(g^n(z))
&= \lim\limits_{r\to\infty}f^{r-1}(l)\\
&=\alpha(\neq\infty).
\end{split}
\end{equation}
	   Therefore, every subsequence will also converge to \(\alpha\). In particular,\(\lim\limits_{n\to\infty}f^i(g^j(z))=\alpha(\neq \infty)\),which is a contradiction.\\
	   Case 2: \(\alpha\) is infinity.\\
	   i.e.,\(\lim\limits_{r\to\infty}f^{r-1}(l)=\infty\). We deduce that
\begin{equation}\notag
\begin{split}
 \lim\limits_{r\to\infty} \lim\limits_{n\to\infty}f^{n+r-1}(g^n(z))
&= \lim\limits_{r\to\infty}f^{r-1}(l)\\
&=\infty.
\end{split}
\end{equation}
	   In particular, \(\lim\limits_{n\to\infty}f^n(g^n(z))=\infty\), i.e., \(z\in I(f\circ g)\) which is a contradiction to our hypothesis.\qedhere

\end{enumerate}
\end{proof}

Recall that the postsingular set of an entire function $f$ is defined as 
\[\mathcal P(f)=\overline{\B(\bigcup_{n\geq 0}f^n(\text{Sing}(f^{-1}))\B)}.\]
Also  an entire function $f$ is called \emph{postsingularly bounded} if the postsingular set $\mathcal P(f)$ is bounded.
The following result  shows that postsingularly bounded entire functions are closed under composition:
\begin{theorem}\label{sec2,thm8}
Let $f$ and $g$ be permutable transcendental entire functions of bounded type. If $f$ and $g$ are postsingularly bounded, then so is $f\circ g.$
\end{theorem}

\begin{proof}
Denote  $A=\text{Sing}(f^{-1}), B=\text{Sing}(g^{-1})$ and $C=\text{Sing}(f\circ g)^{-1}.$ It suffices to show that for all $w\in C, (f\circ g)^{n}(w)$ is bounded as $n\to\ity.$ Suppose that  for some $\al\in C, (f\circ g)^{n}(\al)\to\ity$ as $n\to\ity,$ i.e., $\al\in I(f\circ g).$ As $C\subset A\cup f(B)$ \cite{berg2}, therefore, either $\al\in A$ or $\al\in f(B).$ Also as $I(f\circ g)\subset I(f)\cup I(g),$ 
 accordingly we have the following four cases:
\begin{enumerate}
\item[(i)]   $\al\in A$ and $\al\in I(f).$ In this case, we arrive at   a contradiction to postsingular set of $f$  being bounded.\\
\item[(ii)] $\al\in A$ and $\al\in I(g).$ As $g$ is of bounded type, so $I(g)\subset J(g),$ and as $f\circ g=g\circ f,$ therefore $J(f)=J(g)$ \cite{dinesh1}. We obtain $\al\subset A\cap J(f).$ Since $\al\in\mathcal P(f)$ and $f$ being postsingularly bounded implies $\mathcal P(f)\subset F(f),$ i.e., $\al\in F(f)$ a contradiction.\\
\item[(iii)] $\al\in f(B)$ and $\al\in I(f).$ As $f$ is of bounded type, $I(f)\subset J(f)$ and as argued above $J(f)=J(g).$ Also as $B\subset\mathcal P(g)\subset F(g),$ implies $f(B)\subset f(F(g)).$ This further implies 
\begin{equation}\notag
\begin{split}
f(B)\cap J(g)
& \subset f(F(g))\cap J(g)\\
& \subset F(g)\cap J(g)\\
& =\emptyset,
\end{split}
\end{equation} a contradiction.\\
\item[(iv)] $\al\in f(B)$ and $\al\in I(g).$ As $\al\in f(B),$ so $\al=f(\be)$ for some $\be\in B.$ As $f(\be)\in I(g),$  so $g^n(f(\be))\to\ity$ as $n\to\ity.$ Since $f$ and $g$ are permutable,  $f(g^n(\be))\to\ity$ as $n\to\ity.$ As $f$ is an entire function, $g^n(\be)$ must tend to $\ity$ as $n\to\ity,$ i.e, $\be\in I(g),$ a contradiction to postsingular set of $g$ being bounded. 
\end{enumerate}
Thus in all the four cases we arrive at a contradiction. Hence $f\circ g$ is postsingularly bounded.
\end{proof}
\begin{remark}\label{sec2,rem5}
It can be seen by  induction  that if $g_1,\ldots, g_n$ are permutable entire functions of bounded type which are postsingularly bounded, then so is $g_1\circ\cdots\circ g_n.$
\end{remark}

In \cite{SZ} it was shown that if $f$ is an exponential map, that is, $f= e^{\la z},\,\la\in\C\setminus\{0\},$ then all components of $I(f)$ are unbounded, i.e., Eremenko's conjecture \cite{e1} holds for exponential maps.
Furthermore, in \cite{R1} it was shown that if $f$ is an entire function of bounded type for which all singular orbits are bounded (i.e., $f$ is postsingularly bounded), then each connected component of $I(f)$ is unbounded, providing a partial answer to a conjecture of Eremenko \cite{e1}.
Theorem \ref{sec2,thm8}, in particular, establishes that for a family $\mathcal F$ of permutable entire functions (i.e., any two members in the family are commuting) of bounded type which are postsingularly bounded, each connected component of $I(f)$, $f\in\mathcal F$ is unbounded, providing a partial answer to a conjecture of Eremenko \cite{e1}.

Recall that an entire function $f$ is called \emph{hyperbolic} if the postsingular set $\mathcal P(f)$ is a compact subset of $F(f).$
For instance, $e^{\la z},\,0<\la<\frac{1}{e}$ are examples of hyperbolic entire functions. 
It follows from  result \cite[Theorem 1.1]{R1} that if $f\in\mathcal B$ is hyperbolic, then all components of $I(f)$ are unbounded.\\
The following result  shows that hyperbolic entire functions are closed under composition:
\begin{theorem}\label{thm5'}
Let $f$ and $g$ be hyperbolic entire functions of bounded type satisfying $f\circ g=g\circ f.$ Then $f\circ g$ is hyperbolic.
\end{theorem}

\begin{proof}
It can be easily seen that 
\[\mathcal P(f\circ g)\subset\overline{\B(\bigcup_{k\geq 0}f^k(\mathcal P(g))\bigcup(\bigcup_{m\geq 0}g^m(\mathcal P(f)))\B)}.\]
As $f$ and $g$ are hyperbolic, $\mathcal P(f)$ and $\mathcal P(g)$ are compact subsets of $F(f)$ and $F(g)$ respectively. Also as $f$ and $g$ are permutable entire functions of bounded type we have $F(f)=F(g)$ \cite{dinesh1}. As $f$ is of bounded type, for $z\in F(f)$, $f^n(z)$ stays bounded as $n\to\ity$ \cite{EL}. It follows that $\bigcup_{k\geq 0}f^k(\mathcal P(g))$ is bounded subset of $F(f).$ On similar lines, $\bigcup_{m\geq 0}g^m(\mathcal P(f))$ is bounded subset of $F(f).$ Therefore, $\overline{\B(\bigcup_{k\geq 0}f^k(\mathcal P(g))\bigcup(\bigcup_{m\geq 0}g^m(\mathcal P(f)))\B)}$ is a compact subset of $F(f)=F(f\circ g)$ \cite{dinesh1}. Hence, $\mathcal P(f\circ g)$ is a compact subset of $F(f\circ g)$ and the result gets established.
\end{proof}

\begin{remark}\label{sec2,rem6}
It can be seen by  induction  that if $g_1,\ldots, g_n$ are permutable entire functions of bounded type which are hyperbolic, then so is $g_1\circ\cdots\circ g_n.$
\end{remark}

Theorem \ref{thm5'}, in particular, establishes that for a family $\mathcal G$ of permutable entire functions of bounded type which are hyperbolic, each connected component of $I(g)$, $g\in\mathcal G$ is unbounded, providing a partial answer to a conjecture of Eremenko \cite{e1}.

The following result considers backward invariance of the escaping set under the composite entire function:
\begin{theorem}\label{sec2,thm5}
Let $f$ and $g$ be   transcendental  entire functions 
 satisfying $f\circ g=g\circ f.$ Then
\begin{enumerate}
\item[(1)] $I(f)$ and $I(g)$ are backward invariant under $f\circ g;$
\item[(2)] If $w\notin I(f)$ then $g(w)\notin I(f);$
\end{enumerate}
\end{theorem}

\begin{proof}
\begin{enumerate}
\item [(1)] 
Suppose $w\in\C$ such that $f\circ g(w)\in I(f).$ Then $f^n(f\circ g(w))\to\ity$ as $n\to\ity,$ i.e., $g(f^{n+1}(w))\to\ity$ as $n\to\ity.$ As $g$ is an entire function, this implies that $f^{n+1}(w)$ must tend to $\ity$ as $n\to\ity$ and hence $w\in I(f).$ On similar lines, one can show that $I(g)$ is backward invariant under $f\circ g.$ This completes the proof of the result.
\item [(2)] If $g(w)\in I(f),$ then $f^n(g(w))\to\ity$ as $n\to\ity$, i.e., $g(f^n(w))\to\ity$ as $n\to\ity.$ As $g$ is an entire function, this implies $f^n(w)$ must tend to $\ity$ as $n\to\ity$ contrary to our assumption and hence proves the result.\qedhere
\end{enumerate}
\end{proof}

Recall that if $g$ and $h$ are  entire functions and  $f:\C\to\C$ is a non-constant continuous function such that $\displaystyle{f\circ g=h\circ f,}$
then we say that $g$ and $h$ are semiconjugated (by $f$) and $f$ is called a semiconjugacy. 
This definition first appeared in \cite{berg5} where it was shown that for given entire functions $g$ and $h$, there are at most countable number of entire functions $f$ such that $f\circ g=h\circ f.$
\begin{theorem}\label{sec2,thm1}
If $g$ and $h$ are semiconjugated by $f$ then $f^{-1}(I(h))\subset I(g).$
\end{theorem}

\begin{proof}
Let $w\in f^{-1}(I(h)).$ Then $h^n(f(w))\to\ity$ as $n\to\ity,$ which in turn implies $f(g^n(w))\to\ity$ as $n\to\ity.$ But this gives $g^n(w)$ must tend to $\ity,$ i.e., $w\in I(g)$ and hence the result.
\end{proof}

The particular case when $f$ is entire and $h=g$ gives
\begin{corollary}\label{sec2,cor1}
Let $f$ and $g$ be   transcendental  entire functions satisfying $f\circ g=g\circ f.$ Then $g(I(f))\supset I(f).$
\end{corollary}
The result can be extended up to closure of the escaping set.
\begin{theorem}
Let $f$ and $g$ be entire functions and $f\circ g=g\circ f.$  Then $\overline{I(f)}\subset g(\overline{I(f)}).$
\end{theorem}

\begin{proof}
Suppose $w\in\overline{I(f)}.$ Then there exist a sequence $\{z_n\}\subset I(f)$ such that $z_n\to w$ as $n\to\ity.$ Denote a continuous branch of $g^{-1}$ by $g^{-1}$ itself, we get that $g^{-1}(z_n)\to g^{-1}(w)$ as $n\to\ity.$ Using Corollary  \ref{sec2,cor1} we get $g^{-1}(z_n)\subset I(f)$ and so $g^{-1}(w)\in\overline{I(f)}$ and hence the result.
\end{proof}


\section{some properties of the map $f(z)=z+1+e^{-z}$}\label{sec3}

In this section we explore some dynamical properties of the map $f(z)=z+1+e^{-z}$. Throughout this section, by $f$ and $g$ we will denote the map $f(z)=z+1+e^{-z}$ and $g(z)=f(z)+2\pi i,$ unless otherwise mentioned.
We show  in the next example that in general, for two permutable entire functions $f$ and $g$ (where $f(z)=z+1+e^{-z}$ ), $I(f\circ g)$ may not be contained in $I(f)\cap I(g).$ The function $f$ appeared first in a paper of Fatou \cite{F2}. In this paper, Fatou discussed the iterative behavior of transcendental entire functions and  established some fundamental properties of the Fatou and Julia sets. He further posed some stimulating fundamental questions about the topological properties of the Julia set of a transcendental entire function which led to  further developments. 

\begin{example}\label{sec2,eg1}
Let $f(z)=z+1+e^{-z}$ and $g(z)=f(z)+2\pi i.$ It can be easily seen that $f\circ g=g\circ f.$ $F(f)$ has a \emph{Baker domain} containing the right half plane $\mathcal H_r=\{z\in\C: \Re(z)>0\}$ and so $f$ is not of bounded type. Observe that
\begin{equation}\label{eq1}
 (f\circ g)^n(z)=f^{2n}(z)+2n\pi i 
\end{equation}
and
\begin{equation}\label{eq2}
(g\circ f)^n(z)=f^{2n}(z)+2n\pi i
\end{equation}
Let $Fix(f)$ denote the fixed point set of $f.$ Then $Fix(f)=\{z\in\C: z=(2k+1)i\pi, k\in\Z\}.$ For each $z\in Fix(f), z\notin I(f).$  Also for  $z\in Fix(f),$ we have  
\begin{equation}\label{eq3}
\begin{split}
(f\circ g)^n(z)
&= (2k+1)i\pi+2n\pi i\\
&\to\ity\;\text{as}\;n\to\ity.
\end{split}
\end{equation}
which implies that $z\in I(f\circ g).$ Moreover, as $g^n(z)=f^n(z)+2n\pi i,$ we get for $z\in Fix(f), z\in I(g).$ Hence $I(f\circ g)\nsubseteq I(f)\cap I(g).$
\end{example}

We now establish some invariance properties of $I(f)$ and $I(g)$  for the pair of commuting functions $f(z)=z+1+e^{-z}$ and $g(z)=f(z)+2\pi i.$ Observe that for $l, p\in\N,$ $f$ satisfies $f^l(z+2p\pi i)=f^l(z)+2p\pi i.$  

\begin{proposition}\label{sec2,prop1}
For the pair of commuting functions $f(z)=z+1+e^{-z}$ and $g(z)=f(z)+2\pi i,$ $g(I(f))\subset I(f)$ and $f(I(g))\subset I(g).$
\end{proposition}

\begin{proof}
We first show that $g(I(f))\subset I(f).$ To this end, let $z\in I(f).$ Then for $n\in\N$
\begin{equation}\label{eq4}
\begin{split}
f^n(g(z))
&=f^n(f(z)+2\pi i)\\
&=f^{n+1}(z)+2\pi i\\
&\to\ity\;\text{as}\;n\to\ity
\end{split}
\end{equation}
which implies $g(z)\in I(f)$ and hence the result. We now show that $f(I(g))\subset I(g).$ Let $w\in I(g).$ Consider for $k\in\N$
\begin{equation}\label{eq5}
\begin{split}
g^k(f(w))
&=f(g^k(w))\\
&=g(g^k(w))-2\pi i\\
&\to\ity\;\text{as}\;k\to\ity
\end{split}
\end{equation}
which implies $f(w)\in I(g)$ and hence the result.
\end{proof}
\begin{remark}\label{sec2,rem1}
Using Corollary \ref{sec2,cor1} we get $I(f)$ is completely invariant under $g$ and $I(g)$ is completely invariant under $f.$ Also for each transcendental entire function $h,$ $I(h)$ is completely invariant under $h$ we get that $I(f)$ is completely invariant under $f\circ g.$ On similar lines, $I(g)$ is completely invariant under $f\circ g.$
\end{remark}
An immediate consequence of Proposition \ref{sec2,prop1} is
\begin{corollary}\label{sec2,cor2}
Let $f$ and $g$ be as in Proposition \ref{sec2,prop1}. Then $g\overline{(I(f))}\subset \overline{(I(f))}$ and $f\overline{(I(g))}\subset \overline{(I(g))}.$
\end{corollary}
It is not true in general that for an entire function $h$ if $F(h)\cap AV(h)=\emptyset,$ then all components of $F(h)$ are bounded. The entire function $f(z)=z+1+e^{-z}$ serves as an example for which $F(f)$ is disjoint from the set of asymptotic values $AV(f)$, but all components of $F(f)$ are unbounded. In fact, from Example \ref{sec2,eg1}, $F(f)$ has a Baker domain containing the right half plane $\mathcal H_r.$
Observe that the critical values of $f$ is the set $CV(f)=\{2+2k\pi i: k\in\Z\}$ so that Sing($f^{-1}$) is unbounded.

Recently  \cite[Theorem 1.2]{berg7}  a  characterization was given for boundedness of Fatou components of a hyperbolic entire function of bounded type. 
\begin{theorem}\cite[Theorem 1.2]{berg7}\label{sec2,thm6}
Let $h$ be a hyperbolic entire function of bounded type. Then the following are equivalent:
\begin{enumerate}
\item[(i)] every component of $F(h)$ is bounded;
\item[(ii)] h has no asymptotic values and every component of $F(h)$ contains at most finitely many critical points.
\end{enumerate}
\end{theorem}

The entire functions $f$ and $g$ serves as an example of a pair of commuting functions which satisfies $Fix(f)\cap I(g)\neq\emptyset$ where $Fix(f)$ denotes fixed point set of $f.$ As computed in Example \ref{sec2,eg1}, $Fix(f)=\{z\in\C: z=(2k+1)i\pi, k\in\Z\}.$ We assert that $Fix(f)\subset I(g).$ To this end, let $z\in Fix(f).$ Then $z=(2m+1)i\pi$ for some $m\in\Z.$ Consider 
\begin{equation}\label{eq6}
\begin{split}
g^n(z)
&=f^n(z)+2n\pi i\\
&=(2m+1)i\pi+2n\pi i\\
&\to\ity\;\text{as}\;n\to\ity
\end{split}
\end{equation}
so that $z\in I(g)$ which proves the assertion and hence the result.

It can be easily seen that for an entire function $h,$ $int(I(h))\subset F(h)$ where $int(S)$ denotes interior of a subset $S$ of a topological space $X.$ For obvious reasons, if an entire function $h$ is of bounded type then $int(I(h))=\emptyset.$ However, if $h\notin\mathcal B$ then $int (I(h))$ can be non empty. Consider $f(z)=z+1+e^{-z}.$ As observed in Example \ref{sec2,eg1}, $F(f)$ has a Baker domain containing the right half plane $\mathcal H_r.$ In particular, $\mathcal H_r\subset I(f)$ and hence $int(I(f))\neq\emptyset.$ This fact was also observed in \cite{e1}.

Recall that the order of a transcendental entire function $h$ is defined as 
\[\rho_h=\limsup_{r\to\ity}\frac{\log \log M(r, h)}{\log r}\]
where 
\[M(r, h)=\max\{|h(z)|: |z|\leq r\}.\]
For an entire function of finite order the number of critical values and asymptotic values are related. This is expressed in the following result

\begin{lemma}\cite[Corollary 3]{berg8}\label{sec2,lem1}
If an entire function $h$ of finite order $\rho_h$  has only finitely many critical values, then it has at most $2\rho$ asymptotic values.
\end{lemma}

However, the function $f(z)=z+1+e^{-z}$ serves as an example of an entire function of finite order viz. $1$ having infinitely many critical values (as computed before $CV(f)=\{2+2k\pi i: k\in\Z\}$) but no finite asymptotic value (or at most 2 asymptotic value), so that the conclusion of Lemma \ref{sec2,lem1} still holds for an entire function of finite order having  infinitely many critical values.



\end{document}